\documentclass[12pt]{amsart}
\usepackage[margin=1in]{geometry}
\newtheorem{theorem}{Theorem}[section]
\newtheorem{lemma}{Lemma}[section]

\newtheorem{prop}{Proposition}[section]
\theoremstyle{remark}
\newtheorem{remark}{Remark}[section]
\theoremstyle{definition}
\newtheorem{definition}{Definition}[section]

\begin{document}

\title[Random Stability of Random Variables]{Random Stability of Random Variables}

\author{Andrey Sarantsev}
\address{University of Nevada, Reno, Department of Mathematics and Statistics}
\email{asarantsev@unr.edu}


\begin{abstract}
For a random variable $N = 0, 1, 2, \ldots$ we study the following question: When does   the sum of $N$ many independent and identically distributed copies of a random variable $X$ have the same law a a nontrivial rescaling of $X$? We show that such $N$-stable random variable exists if and only $1 < \mathbb E[N] < \infty$. Under an additional assumption $\mathbb E[N\ln N] < \infty$, we describe all $N$-stable $X$. We also study a converse problem: For a given $X \ge 0$ with $\mathbb E[X] = 1$, we study the set of all $N$ such that $X$ is $N$-stable. Distributions of $N$ form a commuting semigroup with respect to composition of probability generating functions. 
\end{abstract}

\maketitle

\thispagestyle{empty}

\section{Introduction}

\subsection{Main concepts} 
A real-valued random variable $X$ (or its distribution on the real line) is called {\it stable} if for every $n = 2, 3, \ldots$ there exist constants $a_n > 0$ and $b_n$ such that for independent identically distributed copies $X_1, X_2, \ldots, X_n$ of $X$, we have: 
$$
X_1 + \ldots + X_n = a_nX + b_n.
$$
If $b_n = 0$, this random variable or its distribution is called {\it strictly stable}. The theory of stable and strictly stable distributions is classic and well-developed, see the monographs \cite{STbook, SteutelBook}; and a less known but useful book \cite[Chapter 2]{KlebanovBook}; see the discussion in Section 2 below. These distributions are described via their characteristic functions (Fourier transforms) $\mathbb E[e^{\mathrm{i}uX}]$, where $\mathrm{i} = \sqrt{-1}$. 

Applications of strictly stable distributions include: (a) Actuarial mathematics, with $n$ the number of claims and $X_k$ the size of the $k$th claim; then the overall size of all claims has the same distribution as each claim, up to scaling; (b) Quantitative finance, with $X_k$ the stock market log returns on day $k$, and $n$ the number of days; then the cumulative log return for $n$ days is given by the sum $X_1 + \ldots + X_n$; (c) Stable distributions can serve as scaling limits of sums of IID random variables, generalizing the Central Limit Theorem; (d) Increments of stable L\'evy processes, which have stationary independent increment and are also invariant with respect to scaling. 

The topic of this article is {\it random stability} or {\it $N$-stability.} Definition~\ref{def:main-RV} is a generalization of strictly stable distributions $Q$, which are $N$-stable in terms of this Definition~\ref{def:main-RV} for a constant $N = n$ for each $n = 2, 3, \ldots$  

\begin{definition}
Take a random variable $N$ with values $0, 1, 2, \ldots$ with distribution $P$. A probability measure $Q$ on $\mathbb R$, or, equivalently, a random variable $X \sim Q$, is called {\it $N$-stable} or, equivalently, {\it $P$-stable}, if there exists a $c > 0$ such that the following equality in law holds: 
\begin{equation}
\label{eq:main-RV}
X_1 + \ldots + X_N \stackrel{d}{=} cX,
\end{equation}
where $X_1, X_2, \ldots \sim Q$ are independent copies of $X$, independent of $N$. 
\label{def:main-RV}
\end{definition}

In \cite[Definition 3]{Klebanov}, the term {\it strictly $N$-stable} is used. Alternatively, the article \cite{B} uses the term {\it branching stability} to stress connection with branching processes. Recall that a classic (Galton-Watson) discrete-time branching process $(X_0 = 1, X_1, X_2, \ldots)$ is constructed as follows: We fix a distribution $P$ on $\{0, 1, 2, \ldots\}$. At each step $n$, we treat the value $X_n$ of the process as the population, and generate a random number of offsprings (distributed as $Q$) for each member of this population. For all members, the numbers of their offsprings are independent. Classic monographs on branching processes include \cite{Harris} and a more recent book \cite{AthreyaBook}. We will see in our article that this connection is indeed critical. Indeed, the classic result (see Section 2 below) is: Under simple conditions, if $c$ is the mean of $Q$, then $X_n/c^n$ weakly converges to a $Q$-stable distribution as $n \to \infty$.  

Importantly, we can rewrite the main equation from Definition~\ref{def:main-RV}. Taking the characteristic function (Fourier transform) $f(u) = \mathbb E[e^{iuX}]$ of $X$, and the probability generating function $\varphi$ of $N$: $\varphi(s) = \mathbb E[s^N]$, we rewrite this as the {\it Poincare functional equation}:
\begin{equation}
\label{eq:main-ChF}
\varphi(f(u)) = f(cu),\quad u \in \mathbb R.
\end{equation}
If $X \ge 0$, then the same equation holds for the Laplace instead of the Fourier transform. Taking the inverse of this Laplace transform, we can rewrite this Poincare equation, in turn, as another functional equation: {\it Schr\"oder equation}, see \cite{Kuczma}. 

\subsection{Literature review on random stability} For random stability, there has been a lot of research starting from the 1990s, as well as some earlier research in relation to the classic theory of branching processes. See \cite{Koz1} for a survey of geometric stable distributions ($N$-stable for geometric $N$), and a more detailed description in \cite[Chapter 9]{KKRbook}. A classic application for classic stable and random stable laws is asset returns, see \cite[Remark 2.2]{Koz1} and references therein; also \cite[Chapter 4]{KlebanovBook}, page 157. For $N$-stable random variables with general $N$, see \cite[Chapter 4]{KlebanovBook}, and an important short note \cite{Anna}. Also, see  articles \cite{Stable, Zolotarev}.

\subsection{Our contributions} However, existing literature is somewhat fragmented. Equivalent formulations have appeared in different forms, and this article aims to unify and extend them. Theorems~\ref{thm:main}, ~\ref{thm:repr}, and~\ref{thm:commute} are our main results. We could not find Theorem~\ref{thm:main} in other articles. The closest was \cite[Corollary 1.2]{Liu1998} with (in their notation) $A_i = 1/c$. This discusses a generalization of random stability to so-called {it smoothing problems}. But their proof applies only to nonnegative $N$-stable $X$. Theorem~\ref{thm:repr} was shown in \cite{Anna} but under additional complicated conditions.  Theorem~\ref{thm:commute} was shown in special cases, see \cite{GnedenkoBook}. 

\subsection{Organization of the article}
Section~\ref{sec:main} contains main results: Theorems~\ref{thm:main},~\ref{thm:repr}, and~\ref{thm:commute}, as well as most lemmas. In Section~\ref{sec:semigroup}, we fix a distribution $Q$ on the real line and study $G(Q)$, which is the set of distributions $P$ on $0, 1, 2, \ldots$ for which $Q$ is $P$-stable. We also discuss two equivalent definitions of random stability.  The last section is devoted to proofs. 

\section{Main Results}
\label{sec:main}

\subsection{Existence results} Theorem~\ref{thm:main} is the main result of this article. As mentioned earlier, we could not find it in the literature, despite its importance. 

\begin{theorem} Assume $N$ is not identically zero or one. There exists an $N$-stable random variable $X$ which is not identically zero if and only if $1 < \mathbb E[N] < \infty$.  
\label{thm:main}
\end{theorem}

\begin{remark} 
Consider a {\it Sibuya distribution} with probability generating function $G_p(s) = 1 - (1 - s)^p$ for $p \in (0, 1)$. It was introduced in \cite[formula (2)]{S} with (in their notation) $\beta = \gamma = 1$. See also a more recent article \cite{Sibuya} with an extension with mass at zero. This distribution has infinite mean. By Theorem~\ref{thm:main}, there exists no Sibuya-stable random variable. 
\label{rmk:Sibuya}
\end{remark}

The following is a well-known uniqueness result, which can be found in the classic monograph on branching processes: \cite[Chapter I]{AthreyaBook}, Section 10, Theorem 2. Let us impose an additional assumption 
\begin{equation}
\label{eq:logN}
\mathbb E[N\ln N] < \infty
\end{equation}
with the standard convention that $0 \ln 0 := 0$. Of course, ~\eqref{eq:logN} implies $\mathbb E[N] < \infty$. 

\begin{prop} 
Assume $\mathbb E[N] > 1$. Under the assumption~\eqref{eq:logN}, there exists one and only one $N$-stable $X \ge 0$ (up to equality in law) such that $\mathbb E[X] = 1$. For this $X$, we have $c = \mathbb E[N]$ in~\eqref{eq:main-RV}. 
\label{prop:n-ln-n}
\end{prop}

We stress that even for $N$ satisfying $\mathbb E[N] > 1$ and~\eqref{eq:logN}, there exist $N$-stable $X$ with $\mathbb E[X] = \infty$. See the discussion below where we classify all $N$-stable $X$ for a given $N$. 

\begin{remark}
The following result was proved in \cite[Chapter I]{AthreyaBook}, Section 10, Theorem 3. Let $(N_k)$ be a branching process with $N$ offsprings and $c = \mathbb E[N] \in (1, \infty)$. This case is called {\it supercritical}. There exists a sequence $(C_k)$ of positive numbers such that 
\begin{equation}
\label{eq:conv}
\frac{C_{k+1}}{C_k} \to c,\quad \frac{N_k}{C_k} \stackrel{a.s.}{\to} X,
\end{equation}
where $X \ge 0$ is an $N$-stable random variable. There are two cases, when $c = \mathbb E[N] \in (1, \infty)$; see also discussion in \cite{Bingham1988}. For a conceptual proof, see \cite{Concept}. For other proofs, see \cite[Theorem 3.1]{Seneta1968},   \cite[Theorem 4.3]{Seneta1969}.
\begin{itemize}
\item If $\mathbb E[N\ln N] < \infty$, we can simply take $C_k = c^k$. Then $\mathbb E[X] = 1$.
\item If $\mathbb E[N\ln N] = \infty$, we need to take $C_k \ne c^k$, and $\mathbb E[X] = \infty$. 
\end{itemize}  
\label{rmk:conv}
\end{remark}

\begin{definition} We call such $X$ from Proposition~\ref{prop:n-ln-n} the {\it standard $N$-stable random variable}, and the distribution of $X$ is called the {\it standard $N$-stable distribution}. 
\end{definition}

\subsection{Strictly stable distributions on the real line} Recall a well-known classification of stictly stable distributions $Q$, see \cite[Chapter 1]{STbook} or \cite[Chapter V]{SteutelBook}. Its characteristic function (Fourier transform) 
$$
\int_{\mathbb R}e^{\mathrm{i}ux}\,Q(\mathrm{d}x) = e^{-g(u)},\quad \mathrm{i} =  \sqrt{-1},
$$
can be represented as
\begin{equation}
\label{eq:G}
g(u) = (\beta + \gamma\mathrm{i}\,\mathrm{sgn}(u))|u|^{\alpha}
\end{equation}
with $\alpha \in (0, 2]$, $\beta > 0$. Special cases:
\begin{itemize}
\item if $\alpha = 2$, then $\gamma = 0$ and $Q = \mathcal N(0, 2\beta)$;
\item if $\alpha = 1$, then $\gamma$ is unrestricted and $Q$ is shifted Cauchy.
\end{itemize}
Of course, we also have the trivial case $\beta = \gamma = 0$, when $Q = \delta_0$.

Denote the set of such functions $g$ as $\mathcal G$. Such distribution $Q$ is also called {\it strictly $\alpha$-stable} to stress its dependence upon the {\it index} $\alpha$. Note that the function $g$ satisfies 
\begin{equation}
\label{eq:homo}
g(cu) = c^{\alpha}g(u),\quad c > 0,\quad u \in \mathbb R. 
\end{equation}
A standard reference is the classic monograph \cite[Chapter 1]{STbook}: Equivalent Definitions 1.1.1, 1.1.4, 1.1.5, 1.1.6 of stable distributions, and Properties 1.2.6, 1.2.8 of strictly stable distributions. Also, see other monographs: \cite[Chapter 3]{FTbook}: Theorem 5.7.3 and discussions after this; a more recent book \cite[Chapter V]{SteutelBook}: Section 7 for general results about strict stability; Theorem 7.6 for strictly stable symmetric distributions; Theorem 3.5 for positive strictly stable symmetric distributions; Theorem 3.5 for positive strictly stable random variables. Finally, we point out to \cite[Chapter 4]{GnedenkoBook}, section 4.5, page 130. We stress that much of classic research involves {\it stable} not {\it strictly stable} distributions, which arise as weak limits 
\begin{equation}
\label{eq:stable-weak-limits}
\frac{X_1 + \ldots + X_n}{a_n} - b_n,\quad n \to \infty,
\end{equation}
for independent identically distributed $X_1, X_2, \ldots$ But we use strictly stable distributions here in this article, with $b_n = 0$. 

\subsection{Classification of $N$-stable $X$ for given $N$} In Theorem~\ref{thm:repr}, we classify $N$-stable distributions. This replicates the results \cite[Proposition 2.1, Proposition 3.1]{Anna} but without an artificial additional assumption in Remark~\ref{rmk:artificial} below. The $N$-stable laws are represented via a product involving a standard $N$-stable random variable, and an independent strictly $\alpha$-stable random variable. We impose the following conditions:
\begin{equation}
\label{eq:ln-N}
1 < \mathbb E[N]\quad \mbox{and}\quad \mathbb E[N\ln N] < \infty.
\end{equation}

\begin{theorem}
\label{thm:repr} Under the condition~\eqref{eq:ln-N}, any $N$-stable $X$ can be represented as the following product of random variables:
\begin{equation}
\label{eq:repr}
X \stackrel{d}{=} Y^{1/\alpha}Z,
\end{equation}
where $Y$ is the standard $N$-stable random variable, and $Z$ is a strictly $\alpha$-stable random variable independent of $Y$, for some $\alpha \in (0, 2]$. The characteristic function of $X$ can be represented as 
\begin{equation}
\label{eq:main}
f(u) = \mathbb E\left[e^{\mathrm{i}uX}\right] = L(g(u)).
\end{equation}
Here, $L(u) = \mathbb E[e^{-uY}]$ is the Laplace transform of this standard $N$-stable $Y$. Also, $g \in \mathcal G$ is defined via the characteristic function of $Z$, which is $e^{-g}$. Moreover, if $X$ is $N$-stable with index $\alpha$, then it satisfies~\eqref{eq:main-RV} with 
\begin{equation}
\label{eq:c}
c = (\mathbb E[N])^{1/\alpha}.
\end{equation}
Conversely, any $X$ having representation~\eqref{eq:repr}, or characteristic function~\eqref{eq:main}, is $N$-stable. 
\end{theorem} 

As discussed in the proof of Lemma~\ref{lemma:tech}, the Laplace transform can be defined for $\mathrm{Re}(u) > 0$. We need to extend this Laplace transform to the complex half-plane, since $g$ can be complex-valued.

Separately, we state an important uniqueness lemma.

\begin{lemma} Assume $L_1$ and $L_2$ are Laplace transforms of two probability measures $Q_1$ and $Q_2$ on the positive half-line with mean $1$. Let $g_1, g_2 \in \mathcal G$ be such that $L_1\circ g_1 \equiv L_2 \circ g_2$. Then $L_1 \equiv L_2$ and $g_1 \equiv g_2$.
\label{lemma:unique}
\end{lemma}

\subsection{Removal of artificial assumptions} This result is important for the following subtle reason. Earlier, such representation result was proved by \cite{Anna} under an additional condition on the Laplace transform $L(u) = \mathbb E[e^{-uY}]$ of this standard $N$-stable $Y$. 

\begin{remark}
\label{rmk:artificial}
For any two infinitely divisible characteristic functions $e^{-f}$ and $e^{-g}$ on the real line, $L\circ f \equiv L\circ g$ implies $f\equiv g$. 
\end{remark}

This condition was first introduced in \cite{Szasz} and used for the converse of the {\it transfer theorem.} This theorem is about a transfer from weak limits of sums of independent random variables to weak limits of random sums of these variables, \cite[Chapter 4]{GnedenkoBook}: Section 4.1, Theorem 4.1.2, and references therein. Says \cite{Szasz}: {\it The following condition is presumably true, but so far no one was able to prove or disprove it.} It turns out from our research that we do not, in fact, need to prove this condition! Indeed, from 1973 until now, to the best of our knowledge, no one could prove this in the general form. There were many discussions of this condition in a recent book \cite{GnedenkoBook}, which is an unabridged republication from 1996: \cite[Chapter 4]{GnedenkoBook}, Section 2. To quote: {\it No one has managed either to prove this or to refute it yet.} There, it was proved in several cases. In \cite[Chapter 8]{KKRbook}, it was mentioned but unproved in Remark 8.8.16. 

\subsection{Examples of random stable variables} Here, we consider the three cases of $N$-stable distributions: (a) symmetric; (b) Gaussian; (c) on the half-line.

\begin{definition} The random variable $X$ (or its distribution $Q$) is called a {\it symmetric $N$-stable} if it is $N$-stable and $X \stackrel{d}{=} -X$. 
\end{definition}

\begin{lemma} Under the condition~\eqref{eq:ln-N}, any symmetric $N$-stable $X$ has the following characteristic function 
\begin{equation}
\label{eq:ChF-stable-symm}
f(u) = \mathbb E\left[e^{\mathrm{i}uX}\right] = L(\beta |u|^{\alpha})
\end{equation}
for some $\beta > 0$ and $\alpha \in (0, 2]$, where $L$ is the Laplace transform of the standard $N$-stable random variable $Y$. This $X$ can be represented as 
\begin{equation}
\label{eq:repr-N-stable-symm}
X \stackrel{d}{=} Y^{1/\alpha}Z,\quad \mathbb E\left[e^{\mathrm{i}uZ}\right] = \exp\left(-\beta |u|^{\alpha}\right).
\end{equation}
Here, $Z$ is independent of $Y$, and is symmetric $\alpha$-stable. Conversely, any $X$ having representation~\eqref{eq:repr-N-stable-symm}, or characteristic function~\eqref{eq:repr-N-stable-symm}, is symmetric $N$-stable. 
\label{lemma:symmetry}
\end{lemma}

\begin{definition} Any $N$-stable random variable with index $\alpha = 2$ is called a {\it Gaussian $N$-stable random variable}.
\end{definition}

We stress its similarity with classic Gaussian random variables. In the literature, sometimes a Gaussian $N$-stable $X$ is defined using~\eqref{eq:main-RV} with~\eqref{eq:c} with $\alpha = 2$. An additional restriction 
$\mathbb E[X^2] = \infty$ is introduced in \cite[Definition 2]{Stable}. Confusingly, the opposite restriction $\mathbb E[X^2] < \infty$ was introduced in \cite[Chapter 4]{GnedenkoBook}, see Definition 4.6.1; or \cite[Definition 2]{Klebanov}, or \cite[Chapter 8]{KKRbook}, Definition 8.3.2. But we do not need these restrictions in the definition of a Gaussian $N$-stable random variable, as seen in the following lemma.

\begin{lemma} Under the condition~\eqref{eq:ln-N}, any Gaussian $N$-stable $X$ is symmetric: $X \stackrel{d}{=} -X$. It satisfies $\mathbb E[X^2] < \infty$. Taking $Y$ the standard $N$-stable random variable, with Laplace transform $L$, we can represent $X = \sqrt{Y}Z$, where $Z \sim \mathcal N(0, \sigma^2)$ is independent of $Y$. The characteristic function of $X$ is given by 
$$
f(u) = \mathbb E\left[e^{\mathrm{i}uX}\right] = L(\sigma^2u^2/2).
$$
\label{lemma:Gaussian}
\end{lemma}

The final case is a nonnegative $N$-stable random variable $X$. It is easier to study their Laplace transforms than characteristic functions. 

\begin{lemma} Under the condition~\eqref{eq:ln-N}, a nonnegative $N$-stable $X$ has Laplace transform 
$$
\mathbb E\left[e^{-uX}\right] = L(\beta u^{\alpha})
$$
for the Laplace transform $L$ of the standard $N$-stable random variable $Y$, and for some $\alpha \in (0, 1]$ and $\beta > 0$. This $X$ can be represented as $X \stackrel{d}{=} Y^{1/\alpha}Z$, where $Z$ is the stable positive random variable of index $\alpha$, independent of $Y$, defined via Laplace transform
\begin{equation}
\label{eq:Laplace-alpha}
\mathbb E\left[e^{-uZ}\right] = e^{-\beta u^{\alpha}},\quad u > 0.
\end{equation}
\label{lemma:positive}
\end{lemma}

\begin{remark}
A nonnegative $N$-stable $X$ can have only index $\alpha \le 1$, not 
$\alpha \in (1, 2]$. One can view $Z$ as the subordinator process at a fixed time. 
\end{remark}

\subsection{Geometric stability} The most classic example is for the geometric $N_p$ for $p \in (0, 1)$, with distribution 
\begin{equation}
\label{eq:geometric}
\mathbb P(N_p = n) = p(1-p)^{n-1},\quad n = 1, 2, \ldots;
\end{equation}
and with expected value and the probability generating function
\begin{equation}
\label{eq:geometric-PGF}
\mathbb E[s^N] = \frac{ps}{1 - (1-p)s},\quad \mathbb E[N_p] = \frac1p.
\end{equation}
Then the standard $N_p$-stable distribution of every $p \in (0, 1)$ is the standard exponential with mean $1$: It has density and the Laplace transform 
\begin{equation}
\label{eq:exp-density}
p(x) = e^{-x},\quad x \ge 0;\quad L(u) = \mathbb E[e^{-uY}] = \frac{1}{1+u},\quad u \ge 0;\quad Y \sim \mathrm{Exp}(1). 
\end{equation}
A Gaussian $N_p$-stable distribution is also well-known: The {\it Laplace distribution} with density and the characteristic function
$$
p(x) = 0.5\sigma e^{-\sigma|x|},\quad x \in \mathbb R;\quad f(u) = \frac1{1 + \sigma^2u^2}.
$$
More generally, a symmetric $N_p$-stable distribution is known: This is the {\it Linnik distribution} (named after a Ukrainian mathematician) with characteristic function 
\begin{equation}
\label{eq:Linnik}
f(u) = \frac1{1 + \beta |u|^{\alpha}}.
\end{equation}
with $0 < \alpha \le 2$. See also \cite[Chapter 9]{KKRbook}, Subsection 9.4.2.1, Theorem 9.4.6. This distribution is absolutely continuous on the real line, but its density is explicitly known only for $\alpha = 2$, when it becomes the Laplace distribution. The positive $N_p$-stable distribution has Laplace transform with $\beta > 0$ and 
$\alpha \in (0, 1]$:
\begin{equation}
\label{eq:Kovalenko}
L(u) = \frac1{1 + \beta u^{\alpha}}.
\end{equation}
This is commonly known as the {\it Mittag-Leffler distribution}; see \cite[Chapter 9]{KKRbook}, Subsection 9.4.2.2, Theorem 9.4.9. Also, in \cite[Chapter 2]{GnedenkoBook}, the distribution with the Laplace transform~\eqref{eq:Kovalenko} was named the {\it Kovalenko distribution} after another Ukrainian mathematician, see (3.2) or page 30. Some authors call it {\it positive Linnik distribution}, to stress its similarity with the two-sided Linnik distribution from~\eqref{eq:Linnik}. This distribution is also absolutely continuous, but its density is explicitly known only for $\alpha = 1$ or $\alpha = 0.5$. An entire chapter \cite[Chapter 9]{KKRbook} is devoted to these geometric stable distributions. See also a great survey \cite{Koz1} with large bibliography. Such geometric stable distributions were studied as far back as in 1984 in \cite{Zolotarev} and also in one early article by T. Kozubowski \cite{Koz1994}. Further representation results for the Linnik and the Kovalenko distributions are given in \cite{KorolevZeifman}.

\subsection{The infinitely divisible case} The following example is well known, see \cite[Chapter I, Example 8.5; Chapter V, Example 13.2]{Harris},  \cite[Section 8]{B}, \cite[Section 6]{Stable}, Examples 2 and 4, \cite[Example 1]{Bunge}; \cite[Example 3]{Melamed}, and physical applications in \cite[Section 8]{Harris1948}. Here $X$ is a Gamma random variable with shape $1/k$ and certain rate parameter $\beta > 0$. We remind the readers that it has Laplace transform
\begin{equation}
\label{eq:Gamma}
\mathbb E[e^{-uX}] = (1 + \beta u)^{-1/k},\quad u > 0.
\end{equation}
and $N$ has probability generating function for some $p \in (0, 1)$: 
\begin{equation}
\label{eq:NB}
\varphi(s) = \frac{p^{1/k}s}{(1 - (1 - p)s^k)^{1/k}}.
\end{equation}
And $X$ is $N$-stable for any $p$ and $k$. This is related to the classic geometric-exponential pair from~\eqref{eq:exp-density}. Indeed, since the sum of $k$ i.i.d. Gamma from~\eqref{eq:Gamma} is exponential:
$$
X_1 + \ldots + X_k \sim  \mathrm{Exp}(1).
$$
The probability generating function from~\eqref{eq:NB} corresponds to the random variable $kM$. Here, $M$ is defined by its  probability generating function 
$$
\mathbb E[s^M] = \left(\frac{ps}{1 - (1-p)s}\right)^{1/k}
$$
and is negative binomial with shape $1/k$. The sum of $k$ independent identically distributed copies of this $M$ is, in fact, geometric, it is distributed as $N_p$ from~\eqref{eq:geometric}:
$$
M_1 + \ldots + M_k \stackrel{d}{=} N_p. 
$$
In Lemma~\ref{lemma:inf-div}, we generalize this for general infinitely divisible distributions. We remind the readers the classic definition. For background, see \cite{SteutelBook}. 

\begin{definition} A distribution $Q$ on the real line, or, equivalently, a random variable $X$, is called {\it infinitely divisible} if for each $k = 1, 2, \ldots$, we can represent $X$ as a sum of $k$ independent identically distributed random variables $Y_1, \ldots, Y_k$:
$$
X \stackrel{d}{=} Y_1 + \ldots + Y_k.
$$
\end{definition}

In terms of characteristic functions $f(u)$, or Laplace transforms $L(u)$, or probability generating functions $\varphi(s)$, a necessary and sufficient condition is that for every $k = 2, 3, \ldots$ the $k$th root of this function: $f^{1/k}(u)$, $L^{1/k}(u)$, $\varphi^{1/k}(s)$ must also be a characteristic function, or a Laplace transform, or a probability generating function of some distribution. We note that a characteristic function might be complex-valued. In this case, we use the main complex branch of the $k$th root, which maps $1$ to $1$ (and not to other unit roots, $\exp(\mathrm{i}m\pi/k)$ for $m = 1, \ldots, k-1$). This issue does not arise for probability generating functions or Laplace transforms, which are always real-valued and positive. 

\begin{lemma} Take infinitely divisible random variables $N = 0, 1, 2, \ldots$ and $X$ on the real line. Pick a $k = 2, 3, \ldots$ and decompose 
$$
N = M_1 + \ldots + M_k,\quad X = Y_1 + \ldots + Y_k
$$
into $k$ independent copies of $M$ and $Y$, respectively. Then $X$ is $N$-stable if and only if $Y$ is $kM$-stable. 
\label{lemma:inf-div}
\end{lemma}

\subsection{Commuting distributions} The uniqueness result in Lemma~\ref{lemma:unique} for the representation from Theorem~\ref{thm:repr} allows us to prove the following key statement. First, we consider the {\it composition operation} 
\begin{equation}
\label{eq:commute}
\varphi\circ \psi(s) \equiv \varphi(\psi(s))
\end{equation}
for the two probability generating functions $\varphi$ and $\psi$, or, equivalently, the corresponding distributions. This operation is clearly associative, so the set of all probabiltiy distributions on $\{0, 1, 2, \ldots\}$ is a semigroup under this operation. In terms of the random variables $M$ and $N$ with probability generating functions $\varphi$ and $\psi$, their composition $\varphi\circ\psi$ is also a probability generating function of the random variable $N_1 + \ldots + N_M$, where $N_1, N_2, \ldots$ are independent (of each other and $M$) copies of $N$. The unit element for this composition operation is $\varphi(s) = s$, corresponding to the variable $M = 1$. 

\begin{definition}
We say that the probability generating functions $\varphi$ and $\psi$, or their corresponding distributions, or random variables $M$ and $N$, {\it commute} if 
\begin{equation}
\label{eq:commuting-equation}
\varphi\circ\psi \equiv \psi\circ\varphi.
\end{equation}
\end{definition}

The key result is below. It was shown in other literature in particular cases, but not in this general form. For such results for Gaussian $N$-stable random variables, see \cite[Chapter 4]{GnedenkoBook}, Theorem 4.6.1 and Corollary 4.6.1; \cite[Theorem 1]{Klebanov}, \cite[Chapter 8]{KKRbook}, Theorem 8.3.4 and references therein; \cite[Theorem 2.1]{Melamed}. The property of commutative semigroups was also shown in \cite[Remark 2.4]{Koz1}. This survey \cite{Koz1} is very comprehensive for 1990s. However, they claim that geometric family is unique among such explicit commuting families. We discuss later that this is not true.  

\begin{theorem} Assume $\mathbb E[M] < \infty$, $\mathbb E[N\ln N] < \infty$, and let $X$ be a non-degenerate $N$-stable random variable. Then $M$ and $N$ commute if and only if $X$ is $M$-stable. In this case, $\mathbb E[M\ln M] < \infty$. Moreover, the  standard $N$-stable distribution and the standard $M$-stable distribution coincide.  The index $\alpha$ and the representation of $X$ from Theorem~\ref{thm:repr} for both $M$ and $N$ are the same. 
\label{thm:commute}
\end{theorem}

\begin{remark}
As a corollary of Theorem~\ref{thm:commute}, we get that $M$ and $N$ commute if and only if their standard random stable distributions coincide. 
\end{remark}

\begin{lemma}
\label{lemma:curious}
Assume $\mathbb E[N\ln N] < \infty$. If $X$ and $Y$ are $N$-stable, and $X$ is $M$-stable, then $Y$ is also $M$-stable. In this case, $\mathbb E[M\ln M] < \infty$, and $\mathbb E[N] > 1$, and $\mathbb E[M] > 1$. 
\end{lemma}

Examples of  commuting semigroups include geometric distributions from~\eqref{eq:geometric} and Sibuya distributions from Remark~\ref{rmk:Sibuya}. For each of the two families $(N_p,\, p \in (0, 1))$, their probability generating functions $G_p(s) := \mathbb E[s^{N_p}]$ satisfy 
\begin{equation}
\label{eq:semigroup-property}
G_p\circ G_q = G_q\circ G_p = G_{pq}.
\end{equation}
Recall from Remark~\ref{rmk:Sibuya} that there is no Sibuya-stable random variable.

\section{Composition Semigroups} 
\label{sec:semigroup}

\subsection{Introduction to the problem} Fix a probability measure $Q$ on the real line. In this section, we give an overview of known results and add a few new ones on the following topic: Find the set of all distributions $P$ on the nonnegative integers such that $Q$ is $P$-stable. Let $N \sim P$ and assume $\mathbb E[N \ln N] < \infty$. Then 
from Theorem~\ref{thm:commute} and Lemma~\ref{lemma:curious}, it is sufficient to consider $Q$ which is the standard $N$-stable for any such $N$. This allows to use the Laplace transform $L$ of $Q$. Such $X \sim Q$ satisfies $X \ge 0$ and $\mathbb E[X] = 1$. 

However, we stress that not every $X \ge 0$ can have nontrivial $N$ such that $X$ is $N$-stable. First, to be $N$-stable for at least one $N$ with $1 < \mathbb E[N] < \infty$, the distribution of $X$ must be absolutely continuous  on $(0, \infty)$; see \cite[Chapter I]{AthreyaBook}, Section 10, Theorem 4; Secton 12, Corollary 1; also \cite[Chapter 1]{Harris}, Theorem 8.3. However, this absolute continuity is very far from a sufficient condition for existence of $N$ such that $Q$ is $N$-stable. Recall the pair of negative binomial $N$ and Gamma $X$ random variables from~\eqref{eq:Gamma} and~\eqref{eq:NB} such that $X$ is $N$-stable. Compare it with Lemma~\ref{lemma:Gauss}.

\begin{lemma}
A Gamma random variable with shape parameter $\alpha \ne 1/n$ for $n = 1, 2, \ldots$ is not $N$-stable for any $N$. 
\label{lemma:Gauss}
\end{lemma}

The idea of the proof of Lemma~\ref{lemma:Gauss} is to consider the behavior of the probability generating function of $N$ at $s = 0$. From~\eqref{eq:main-ChF}, which involves the behavior of the Laplace transform of $X$ at infinity. This approach is useful to disprove that $\varphi$ is a probability generating function in other cases. Such distribution $Q$ might have an atom at $0$, that is, have $Q({0}) = \mathbb P(X = 0) > 0$; and this happens if and only if $N$ has an atom at zero: $\mathbb P(N = 0) > 0$. By the way, the case of an atom at zero can be reduced to the case of no atom at zero. See \cite[Chapter I]{AthreyaBook}, Section 12; \cite[Section 9]{B}; \cite[Section 3]{Harris1948}, Definitions 3.1, 3.2; Theorem 3.2. 

\subsection{Analytic description} Recall that $Y$ is a standard $N$-stable random variable. For a given $Y$, let us describe such $N$ analytically. Let $\pi = \mathbb P(Y = 0)$. The Laplace transform $L$ is a one-to-one strictly decreasing function $[0, \infty) \to (\pi, 1]$, and we can find its inverse $L^{\leftarrow}$ (denoted using the arrow to distinguish it from $1/L$): 
\begin{equation}
\label{eq:inverse-Laplace}
L(L^{\leftarrow}(s)) \equiv s,\quad s \in (\pi, 1].
\end{equation}
Recall that the functional equation~\eqref{eq:main-ChF} holds for the Laplace transform $L$ of $Y$ as well as for its characteristic function $f$:
\begin{equation}
\label{eq:main-Laplace}
\varphi(L(u)) = L(cu),\quad u \ge 0.
\end{equation}
Using~\eqref{eq:inverse-Laplace}, we rewrite~\eqref{eq:main-Laplace} as {\it Schroder functional equation}:
$$
L^{\leftarrow}(\varphi(s)) = cL^{\leftarrow}(s),\quad s \in (\pi, 1]. 
$$
In complex analysis, Schroder's equation is well known, see \cite{Shapiro} for a bibliography and a survey of results. We can also rewrite~\eqref{eq:main-Laplace} using~\eqref{eq:inverse-Laplace} as follows: 
\begin{equation}
\label{eq:Bunge}
\varphi(s) = L(cL^{\leftarrow}(s)),\quad s \in (\pi, 1].
\end{equation}
Equation~\eqref{eq:Bunge} is the key formula in a seminal article \cite{Bunge} on random stability. It will play an important role in this article as well.

This defines the probability generating function on the interval $(\pi, 1] \subseteq [0, 1]$. Maybe this is not the entire $[0, 1]$, but even if this interval is smaller, it is large enough to uniquely determine the function $\varphi$. Indeed, this function $\varphi$, as any probability generating function, is analytic in the unit disc $\mathbb D := \{z \in \mathbb C\mid |z| < 1\}$. Therefore, the values of $\varphi$ on this interval uniquely determine it. 

\subsection{The two semigroups} Take any probability measure $Q$ on the half-line with mean $1$. We can define two semigroups:

\begin{itemize}
\item The set $G(Q)$, or, equivalently, $G(X)$, of all probability generating functions 
$\varphi$ (or, equivalently, all distributions $P$ on $\{0, 1, 2, \ldots\}$) such that $Q$ is $P$-stable. This is a semigroup with respect to the composition operation. 
\item The set $S(X)$ or $S(Q)$ of all $c \ge 1$ such that the right-hand side of~\eqref{eq:Bunge} is a probability generating function. This is a semigroup with respect to the standard multiplication operation for real numbers. 
\end{itemize}

For a given $Q$, one of the two possibilities exist. Either there is no $N$ except $N \equiv 1$ such that~\eqref{eq:main-RV} holds. In this case, $G(Q)$ contains only the identity probability generating function $\varphi(s) \equiv s$; and $S(Q) = \{1\}$. Or $Q$ is $N$-stable for some nontrivial $N$. The main formula~\eqref{eq:Bunge} defines a mapping $\mathcal M : S(Q) \to G(Q)$ as follows: 
\begin{equation}
\label{eq:mapping}
\mathcal M : c \mapsto L(cL^{\leftarrow}(\cdot))
\end{equation}

\begin{lemma} The set $G(Q)$ is closed under composition and weak convergence. The set $S(Q)$ is closed under multiplication and under the usual convergence of real numbers. The mapping $\mathcal M$ from~\eqref{eq:mapping} is a bijection, and it preserves the semigroup operation, as well as convergence. If $1$ is the limit point of $S(Q)$, then $S(Q) = [1, \infty)$. 
\label{lemma:limit-point}
\end{lemma}

Lemma~\ref{lemma:finite-moments} is an important result about finite moments for distributions in $G(Q)$. 

\begin{lemma} 
Pick a distribution $Q$ on $[0, \infty)$, and assume $G(Q) \ne \{1\}$ is nontrivial. Then let $N$ be any nontrivial random variable with distribution from $G(Q)$, and let $Y$ be the standard $N$-stable random variable. Then:
\begin{itemize} 
\item $\mathbb E[N\ln N] < \infty$ if and only if $\mathbb E[Y] < \infty$;
\item $\mathbb E[N\ln^{a+1} N] < \infty$ for a fixed $a > 0$ if and only if $\mathbb E[Y\ln^{a}Y] < \infty$;
\item $\mathbb E[N^k] < \infty$ for a fixed $k = 2, 3, \ldots$ if and only if $\mathbb E[Y^k] < \infty$. 
\end{itemize}
\label{lemma:finite-moments}
\end{lemma}

As a corollary, if one distribution $P \in G(Q)$ satisfies $\mathbb E[N\ln^{a+1}N] < \infty$ for a fixed $a \ge 0$, or $\mathbb E[N^k] < \infty$ for a fixed $k = 2, 3, \ldots$ where $N \sim P$, then all other distributions in $G(Q)$ satisfy this property, too. 

\subsection{Examples of semigroups} Existing literature shows various examples of such multiplication semigroups of $S(Q) \subseteq [1, \infty)$. If there is no nontrivial $P$ such that $Q$ is $P$-stable, then $S(Q) = \{1\}$. The other exteme, $S(Q) = [1, \infty)$, happens when there is a continuous-time branching process $N = (N(t),\, t \ge 0)$ such that for every $t \ge 0$, $Q$ is $N(t)$-stable. Background on continuous-time branching processes is provided in \cite[Chapter V]{Harris}. On reparameterization, see a subsection below, and in particular~\eqref{eq:param-p}. In a seminal article \cite{Bunge}, we have $S(Q) = \{1, c, c^2, \ldots\}$ in  \cite[Examples 2, 4]{Bunge}. These are based on discrete-time branching processes not embeddable into continuous-time branching processes, although we have reservations about explanations there. Also, $S(Q) = \{1, 2, 3, \ldots\}$ in \cite[Example 3]{Bunge} for $Q = \delta_1$. Finally, the article \cite{Klebanov} provides an example of $S(Q) = \{1, 4, 9, \ldots\}$ for $Q$ having Laplace transform $(\cosh\sqrt{2u})^{-1}$ for $u \ge 0$, where $\cosh$ is the hyperbolic cosine.

\subsection{Parameterization} Here, we use $c = \mathbb E[N] \ge 1$ (for the standard $X$, we have $\alpha = 1$), and the semigroup operation on $c$ is multiplication. This exactly matches the notation in \cite{Bunge} with $e^t$ instead of $c$, and $t \ge 0$. The corresponding semigroup operation here is addition. However, in other literature, a different parameterization is used: $p \in (0, 1)$. Historically, $N$-stable $X$ arose as weak limits of scaled random sums of random variables. Much like classic stable distributions are weak limits of scaled and shifted sums of a deterministic but growing number of independent identically distributed random variables. Therefore, the notation was used is $\nu_p$ instead of $N$, with 
\begin{equation}
\label{eq:param-p}
p\nu_p \stackrel{d}{\to} X \ge 0,\quad p \in \Xi \subseteq (0, 1),\quad p \downarrow 0;
\end{equation}
where $X$ is, in fact, the standard $\nu_p$-stable random variable for any $p$; and $\Xi$ is a subset of $(0, 1)$ with limit point $0$. The parameter $p$ plays the same role as the parameter $p$ from geometric distributions in~\eqref{eq:geometric}. In other cases, these distributions could be defined only for some $p \in (0, 1)$. 

Some previous authors found it more convenient to impose this limit assumption, and not derive it from scratch. Often, another additional condition was imposed: $\mathbb E[\nu_p] = 1/p$. Normalizing by convergence and expectation at the same time. But this might not be conistent, and we avoid this notation here. The reason is subtle: Match the discussion in Section 1, Remark~\ref{rmk:conv}, with the $p$-notation, and recall the two cases, with $\mathbb E[N\ln N] < \infty$ or $\mathbb E[N\ln N] = \infty$. In both cases, we do have $\mathbb E[N_k] = c^k$.

\begin{itemize}
\item For the case $\mathbb E[N\ln N] < \infty$, or equivalently, $\mathbb E[Z] = 1$, we can let $p = c^{-k}$ and $\nu_p = N_k$. Then 
$$
\mathbb E[\nu_p] = \frac1p,\quad p\nu_p \stackrel{d}{\to} Z,\quad p \downarrow 0.
$$
\item If $\mathbb E[N\ln N] = \infty$, or equivalently, $\mathbb E[Z] = \infty$, then either let $p = c^{-k}$; then $\mathbb E[\nu_p] = 1/p$ but not $p\nu_p \to Z$; or let $p = 1/C_k$, then $p\nu_p \to Z$ but  $\mathbb E[\nu_p] \ne 1/p$. We prefer parameterization $p = 1/C_k$. 
\end{itemize}

In the second case, we cannot have both expectation and convergence assumptions for $\nu_p$. The notation of $\nu_p$ is used in the following literature (the list is not exhaustive): \cite[Section 2]{Melamed}, (1) and (3), both expectation and convergence; \cite[Section 1]{Anna}, only convergence; \cite[page 304]{Stable}, they mentioned only expectation but clearly meant convergence as well, as seen from Theorem 2; \cite[Section 2]{Klebanov}, only expectation, but later convergence is introduced; \cite[Section 2]{Anna} (3) and (4), only convergence; \cite[Chapter 4]{GnedenkoBook}, (6.32), only convergence. Sometimes, commutativity is required or derived in special cases. 

The real reason behind the prevalence of this $p$-notation is that, historically, random stability study to a large extend (although not exclusively) was focused on geometric stability (that is, when $N = N_p$ is geometric from~\eqref{eq:geometric}). In this case, $pN_p$ converges to the standard exponential distribution as $p \downarrow 0$, and $\mathbb E[N_p] = 1/p$. Approximations of sums of $N_p$ many independent identically distributed random variables as $p \downarrow 0$ (similar in spirit to the Central Limit Theorem) were extensively studied. A good exposition is in \cite[Chapter 9]{KKRbook}. 

\begin{lemma}
\label{lemma:semigroup-scaling}
Pick a distribution $Q$ on the half-line. Assume $G(Q)$ is nontrivial. Then the random variable $N_c$ with probability generating function~\eqref{eq:Bunge} with $L$ being the Laplace transform of $Y$ satisfies 
\begin{equation}
\label{eq:param-c}
\frac{N_c}{c} \stackrel{d}{\to} Y,\quad c \to \infty,\quad c \in S(Q), 
\end{equation}
if and only if $\mathbb E[N_c\ln N_c] < \infty$. In this case, $\mathbb E[Y] = 1$.
\end{lemma}

\subsection{Equivalent definitions of random stability} 
There is another definition of random stable variables in the literature, one that is based on weak limits of scaled random sums. Stable and strictly stable distributions  are often defined as weak limits of scaled sums of $N$ independent identically distributed random variables as $N \to \infty$. The same can be done for random stable distributions. See \cite[Section 2]{Anna}, and \cite[Theorems 1, 2, 3]{B}.

Instead of $N \to \infty$, we need a semigroup of probability distributions on nonnegative integers; or, equivalently, a semigroup of probability generating functions. We might as well assume this semigroup is commutative. Although we think one does not really need this assumption, the proof is easy if we impose it. Luckily, even when we start from one $N$ such that $X$ is $N$-stable, such semigroup comes naturally: This is the discrete semigroup generated by $N$; or, equivalently, the set of probability generating functions of $N_k$ for each $k$, where $(N_k)$ is the discrete-time branching process with $N$ offsprings. 

The result below is simple but needs to be included separately. We could not find it explicitly written in the literature. 

\begin{lemma} Take a random variable $Y \ge 0$ with $\mathbb E[Y] = 1$. Assume its semigroup $G(Y)$ is not trivial. Let $N_c$ be the random variable with probability generating function as in~\eqref{eq:Bunge}, with $c \in S(Y)$. Then a random variable  
$X$ is $N$-stable for at least one, and therefore for all, $N$ with distribution in $G(Y)$, if and only if there exists a sequence of independent identically distributed random variables $U_1, U_2,\ldots$ independent of all these $N$, and a function $a : S(Y) \to [0, \infty)$ such that 
\begin{equation}
\label{eq:1}
\frac{U_1 + \ldots + U_{N_c}}{a(c)} \stackrel{d}{\to} X,\quad c \to \infty,\quad c \in S(Y).
\end{equation}
If this is true, then for any such sequence $U_1, U_2, \ldots$ the statement~\eqref{eq:1} is equivalent to~\eqref{eq:2}:
\begin{equation}
\label{eq:2}
\frac{U_1 + \ldots + U_{[c]}}{a(c)} \stackrel{d}{\to} Z,\quad c \to \infty,\quad c \in S(Y).
\end{equation}
In this case, we have~\eqref{eq:repr}, with $Z$ strictly $\alpha$-stable, and $Y$ and $Z$ independent. Moreover, the distributions of $X$ and $Z$ are related as in~\eqref{eq:repr}. 
\label{lemma:equiv}
\end{lemma}

We see that $U_j$ belong to the strict domain of attraction: Recall our discussion above about defining strictly stable distributions as weak limits of scaled sums of independent identically distributed random variables. This is different from classic 
$\alpha$-stable domains of attraction, where adding constants is allowed: Recall~\eqref{eq:stable-weak-limits} and the discussion there. Strict domain of stability include scaling sums of independent indentically distributed random variables without  adding constants. See the discussion in \cite[Section 2]{Anna}.

\section*{Acknowledgements} 

We thank the referee for useful comments which led to catching many misprints and greatly improving the manuscript. We thank our departmental colleague Tomasz Kozubowski for raising the question whether there exists a Sibuya-stable random variable, and another question whether there exist non-geometric and non-Sibuya commuting semigroups; and further useful discussion. We thank Thierry Huillet for pointing out recent articles on continuous-time branching processes with explicit distributions. We thanks Svetlozar Rachev for multiple useful comments on an earlier draft. We thank Lev Klebanov for useful discussion, and for pointing our attention to his book \cite{KlebanovBook}.

\section{Proofs}

First, we present proofs of three main results: Theorems~\ref{thm:main},~\ref{thm:repr},
~\ref{thm:commute}. Next, we present proofs of the twelve lemmas: eleven from the main text plus an additional technical lemma. 

\subsection{Proof of Theorem~\ref{thm:main}} 

This is the longest proof, and we present a short overview. We split it in 8 steps. Step 1  is introductory and involves branching processes and random sums of copies of $X$. Step 2 is $\mathbb E[N] \le 1$. Steps 3--7 are for the case $\mathbb E[N] = \infty$. We symmetrize the distribution $X$ and then use tightness arguments to arrive at a contradiction. Finally, Step 8 mentions the classic case $1 < \mathbb E[N] < \infty$. 

{\it Step 1.} As before, let $\varphi$ be the probability generating function of $N$. Consider a discrete-time branching process $(N_k)$ starting from $N_0 = 1$ and each particle having the number of offsprings distributed as $N$. The probability generating function of $N_k$ is the $k$th composition of $\varphi$:
$$
\varphi_k(s) = \mathbb E\left[s^{N_k}\right] = \varphi(\varphi(\ldots\varphi(s)\ldots)).
$$
For background, we refer to the classic monographs \cite[Chapter 1]{Harris} and \cite[Chapter I]{AthreyaBook}. Applying the main equality $k$ times, we get:
\begin{equation}
\label{eq:k-times}
X_1 + \ldots + X_{N_k} \stackrel{d}{=} c^kZ_1.
\end{equation}
This identity is helpful for the proofs of lack of $N$-stable $X$ for $\mathbb E[N] < 1$ and $\mathbb E[N] = \infty$. The first case is simple, the second is much harder. We always assume $(N_k)$ is independent of $X_1, X_2, \ldots$ 

{\it Step 2.} Assume $\mathbb E[N] \le 1$. From the classic theory of branching processes \cite[Chapter 1]{Harris}, Theorem 6.1 or \cite[Chapter I]{AthreyaBook}, Section 5, Theorem 1, the process $(N_k)$ becomes extinct with probability $1$. That is, $N_k = 0$ for some $k$ almost surely. This immediately implies that the left-hand side of~\eqref{eq:k-times} is equal to zero. Therefore, the right-hand side of~\eqref{eq:k-times} is also zero, which implies $X_1 = 0$ almost surely. This contradiction completes the proof of lack of nontrivial $N$-stable $X$ if $\mathbb E[N] \le 1$. 

{\it Step 3.} Now assume $\mathbb E[N] = \infty$. Define $Y_n := X_n - X'_n$, where $X'_1, X'_2, \ldots$ is yet another sequence of copies of $X$, independent of each other, of $X_1, X_2, \ldots$ and of the branching process $(N_k)$. Then 
\begin{equation}
\label{eq:Y-sum}
Y_1 + \ldots + Y_{N_k} = \left(X_1 + \ldots + X_{N_k}\right) - \left(X'_1 + \ldots + X'_{N_k}\right).
\end{equation}
Dividing~\eqref{eq:Y-sum} by $c^k$, we get:
\begin{align}
\label{eq:Sigma}
\Sigma_k &:= \frac{1}{c^k}\left(Y_1 + \ldots + Y_{N_k}\right) = S_k - S'_k,\\
S_k &:= \frac1{c^k}\left(X_1 + \ldots + X_{N_k}\right),\quad S'_k = \frac1{c^k}\left(X'_1 + \ldots + X'_{N_k}\right).
\end{align}
Note that $S_k$ and $S'_k$ are dependent via $N_k$. Thus we {\it cannot} claim that $\Sigma_k \stackrel{d}{=} X - X'$ for an independent copy $X'$ of $X$. However, we can claim the sequence $(\Sigma_k)$ is tight (in other words, relatively compact, bounded in probability): For any $\varepsilon > 0$, there exists a $K > 0$ large enough so that $\mathbb P(|\Sigma_k| \ge K) \le \varepsilon$ for all $k$. This follows from the tightness of sequences $(S_k)$ and $(S'_k)$, which all have the same distribution, the same as $X$. 

{\it Step 4.} Define the characteristic function of $\Sigma_k$:
\begin{equation}
\label{eq:g-k}
g_k(u) = \mathbb E\left[e^{\mathrm{i}u\Sigma_k}\right]
\end{equation}
It is real-valued, since $\Sigma_k \stackrel{d}{=} -\Sigma_k$. It turns out that this simplifies the proof in a critical way. See \cite[Chapter 3]{FTbook}, Theorem 3.1.2. Applying \cite[Chapter 3]{Stroock}, Lemma 3.1.3, we derive from tightness of the sequence $(\Sigma_k)$ that 
\begin{equation}
\label{eq:tight-conv}
\sup\limits_{k \ge 1}|1 - g_k(u)| \to 0,\quad u \to 0.
\end{equation}
Since $g_k$ is real-valued, and by \cite[Chapter 2]{FTbook}, page 36, $g_k(u) \le 1$, we can simply write this statement~\eqref{eq:tight-conv} without the absolute value. This characteristic function $g_k$ from~\eqref{eq:g-k} can be represented as 
\begin{equation}
\label{eq:g-f}
g_k(u) = \varphi_k(g(u/c^k)),\quad g(u) := \mathbb E\left[e^{\mathrm{i}uY_n}\right] = \mathbb E\left[e^{\mathrm{i}u(X_n - X'_n)}\right] = |f(u)|^2
\end{equation}
is the characteristic function of $Y_n$, see \cite[Chapter 3]{FTbook}, Corollary 2 of Theorem 3.3.1. 

{\it Step 5.} Fix an $a > 1$. From scaling theory for discrete-time branching processes with infinite mean, see for example \cite[Theorem 4.4]{Seneta1969}, the event 
$A = \{N_k/a^k \to \infty\}$ has positive probability. For any $s \in (0, 1)$, we can rewrite
$$
A = \left\{s^{N_k/a^k} \to 0,\quad k \to \infty\right\}
$$
and note $0 \le s^{N_k/a^k} \le 1$ almost surely. Applying the Fatou lemma, we get:
$$
\varphi_k\left(s^{a^{-k}}\right) = \mathbb E\left[s^{N_k/a^k}\right] 
$$
satisfies the upper limit
\begin{equation}
\label{eq:varphi-liminf}
\varlimsup\limits_{k \to \infty}\varphi_k\left(s^{a^{-k}}\right) \le \mathbb E\left[\varlimsup\limits_{k \to \infty}s^{N_k/a^k}\right]\le 0 \cdot \mathbb P(A) + 1\cdot \mathbb P(A^c) < 1. 
\end{equation}

{\it Step 6.} Comparing~\eqref{eq:varphi-liminf} with~\eqref{eq:tight-conv}, we get: For any $u, s \in (0, 1)$ and $a > 1$, there exists a $k_0(u, s, a)$ such that for $k \ge k_0(u, s, a)$ we have: 
\begin{equation}
\label{eq:est}
g(u/c^k) \ge s^{a^{-k}}.
\end{equation}
Take the logarithms in~\eqref{eq:est} and apply the elementary inequality $\ln y \le y - 1$ for $y > 0$:
\begin{equation}
\label{eq:est1}
g(u/c^k) - 1 \ge a^{-k}\ln s.
\end{equation}
Multiplying~\eqref{eq:est1} by $a^k$ and letting $k \to \infty$, we get:
\begin{equation}
\label{eq:final-liminf}
\varliminf\limits_{k \to \infty}\left[a^{k}(g(uc^{-k}) - 1)\right] \ge \ln s.
\end{equation}
In the right-hand side of~\eqref{eq:final-liminf}, the number $s \in (0, 1)$ is arbitrary. Therefore, we can take $s$ as close to $1$ as we wish, to make $\ln s$ negative but as close to zero as we wish. Also, recall that by \cite[Chapter 2]{FTbook}, page 36, we have: $g(uc^{-k}) \le 1$. Applying all this to~\eqref{eq:final-liminf}, we get:
\begin{equation}
\label{eq:lim}
\lim\limits_{k \to \infty}a^{k}(g(uc^{-k}) - 1) = 0.
\end{equation}

{\it Step 7.} Rewrite as second-order discrete difference with step $t_k := 0.5c^{-k}u$:
$$
2(g(uc^{-k}) - 1) = g(-uc^{-k}) + g(uc^{-k}) - 2g(0) = \triangle^{t_k}_2g(0)
$$
where we define $\triangle ^t_2g(u) := g(u - 2t) + g(u + 2t) - 2g(u)$. Clearly, $t_k \to 0$, and letting $a = c^2$ (since $a$ is arbitrary), we rewrite~\eqref{eq:lim} as 
$\triangle_{2}^{t_k}g(0)/t_k^2 \to 0$. Apply \cite[Chapter 2]{FTbook}, Theorem 2.3.1 and conclude: $g''(0) = 0$, thus $\mathbb E[Y^2] = 0$, and $Y = 0$ almost surely, thus $g(u) \equiv 1$. Comparing this with~\eqref{eq:g-f}, we get: $|f(u)| \equiv 1$. Apply \cite[Chapter 2]{FTbook}, Theorem 2.1.4, Corollary 1. We get: $f(u) = e^{\mathrm{i}uX_0}$ and $X = x_0$ almost surely for some $x_0 \in \mathbb R$. But $X$ is $N$-stable. Plugging $X = x_0$ into~\eqref{eq:main-RV}, we get: $x_0N = cx$. Thus either $N = c$, which contradicts $\mathbb E[N] = \infty$; or $x_0 = 0$, which implies $X = 0$ almost surely, but we exclude this trivial case. This completes the proof that there is no $N$-stable $X$ in case $\mathbb E[N] = \infty$. 

{\it Step 8.} The classic case $1 < \mathbb E[N] < \infty$ is well-known and discussed in the Introduction, using scaling limits of dsicrete-time supercritical branching processes. See \cite[Chapter I]{AthreyaBook}, Section 10, Theorems 2 and 3.

\subsection{Proof of Theorem~\ref{thm:repr}} 

{\it Step 1.} We show that any random variable $X$ with characteristic function $f = L \circ g$ and $g \in \mathcal G$ is $N$-stable. Apply~\eqref{eq:main} and~\eqref{eq:main-ChF} with $c = \mathbb E[N]$:
$$
\varphi(f(u)) = \varphi(L(g(u))) = L(cg(u)) = L(g(c^{1/\alpha}u)) = f(c^{1/\alpha}u).
$$
Incidentally, this proves that characteristic functions of the left- and right-hand sides of~\eqref{eq:main-RV} coincide. Thus we have equality in law. 

{\it Step 2.} We show that any random variable $X$ with characteristic function 
$$
f(u) = \mathbb E\left[e^{\mathrm{i}uX}\right] = L(g(u))
$$
for some $g \in \mathcal G$ of index $\alpha$ can be represented as $X = Y^{1/\alpha}Z$, where $Y$ is the standard $N$-stable random variable, and $\mathbb E\left[e^{\mathrm{i}uZ}\right] = e^{-g(u)}$. We simply apply \cite[Proposition 3.1]{Anna} and use their remarks on page 309, (10) -- (12), Remark 1, about the difference between stable and strictly stable. One can also consult \cite[Chapter 1]{STbook}, Section 2 for the latter question. 

{\it Step 3.} Take a characteristic function $f$ of an $N$-stable random variable $X$. We need to prove that this solution to the main equation~\eqref{eq:main-ChF} can be represented as $f = L\circ g$ for some $g \in \mathcal G$. We apply results \cite[Chapter 4]{GnedenkoBook}, Section 4.6. Recall that 
$(N_0 = 1, N_1 = N, N_2, N_3,\ldots)$ is a branching process with $N$ offsprings of each person. Denote $\mathbb E[N] = a$, then $\mathbb E[N_k] = a^k$. By Remark~\ref{rmk:conv}, since $\mathbb E[N\ln N] < \infty$, we have: $a^{-k}N_k \stackrel{d}{\to} X$ with $\mathbb E[X] = 1$ and $X \ge 0$ a.s., with Laplace transform $L$. Also, recall the discussion about parameterization in or below~\eqref{eq:param-p}. Match the notation with \cite[Chapter 4]{GnedenkoBook} and with our subsection on parameterization: 
$$
\Theta = \{a^{-1}, a^{-2}, \ldots\},
$$
with $\nu_{a^{-k}}$ in the notation of Kozubowski which corresponds to $N_{a^{-k}}$ in the notation of \cite[Chapter 4]{GnedenkoBook} becomes $N_k$, and $\nu$ in the notation of Kozubowski which corresponds to $N$ in the notation of \cite[Chapter 4]{GnedenkoBook} becomes $X$. As in the proof of Theorem~\ref{thm:main} (applying the equality in law multiple times), we conclude that if $X$ is $N$-stable, then $X_1 + \ldots + X_{N_k} \stackrel{d}{=} c^kX_1$. Rewrite this using our new notation:
$$
X_{\theta, 1} + \ldots + X_{\theta, N_k} \stackrel{d}{=} X,\quad \theta = a^{-k},\quad X_{\theta, j} := \frac{X_j}{c^k}.
$$
In the notation of \cite[Chapter 4]{GnedenkoBook}, we have (6.33) with $F$ being the CDF of $X$. Actually, our statement is even stronger: We have equality in law, not just convergence in law. Also, we have (6.32) in the same notation. This was discussed in the subsection on parameterization. Therefore, by \cite[Chapter 4]{GnedenkoBook}, Theorem 4.6.3, the formula (6.34) holds with an infinitely divisible $Z$ with the CDF $G$ and with $m(\theta) = [1/\theta] = [a^k]$. Rewrite this conclusion in our original notation:
$$
\frac{X_1 + \ldots + X_{[a^k]}}{c^k} \stackrel{d}{\to} Z.
$$
By the classic theory of stable distributions, this implies $Z$ is strictly stable. And thus $\mathbb E\left[e^{\mathrm{i}uZ}\right] = e^{-g(u)}$ for some $g \in \mathcal G$. By \cite[Chapter 4]{GnedenkoBook}, Theorem 4.6.5, we can write (6.30), which completes the proof of the representation $f = L\circ g$. With this, we proved that a random variable $X$ is $N$-stable if and only if its characteristic function $f$ can be represented as $L\circ g$ with $g \in \mathcal G$, and thus completed the proof of Theorem 2. 

\subsection{Proof of Theorem~\ref{thm:commute}} 

{\it Step 1.} From Theorem~\ref{thm:repr}, the characteristic function $f$ of $X$ can be represented as $f(u) = L(g(u))$, where $L$ is the Laplace transform of the standard $N$-stable $Y$, and $g \in \mathcal G$. We also have~\eqref{eq:main-ChF} with $c = \mathbb E[N]$. Finally, $\mathbb E[Y] = 1$. 

{\it Step 2.} Assume $M$ and $N$ commute. Then their probability generating functions $\varphi$ and $\psi$ satisfy~\eqref{eq:commute}, and therefore
\begin{equation}
\label{eq:all}
\varphi(\psi(L(u))) = \psi(\varphi(L(u))) = \psi(L(cu)).
\end{equation}
The function $\psi\circ L$ is the Laplace transform of the random variable $S = Y_1 + \ldots + Y_M$, where $Y_1, Y_2,\ldots$ are copies of $Y$ independent of each other and of $M$. From~\eqref{eq:all}, we see that $S$ is $N$-stable, and has finite mean $\mathbb E[S] = \mathbb E[Y]\cdot \mathbb E[M] = \mathbb E[M] = b$. By the uniqueness of the standard $N$-stable random variable, $S/b \stackrel{d}{=} Y$. If we write this equality in law in terms of Laplace transforms, we have: $\psi(L(u)) = L(bu)$. Therefore, $Y$ is $M$-stable. Finally, letting $\alpha$ be the index of $g$, we get: 
$$
\psi(f(u)) = \psi(L(g(u))) = L(bg(u)) = L(g(b^{1/\alpha}u)) = f(b^{1/\alpha}u).
$$
This proves $X$ is $M$-stable as well. 

{\it Step 3.} Conversely, if $X$ is both $M$-stable and $N$-stable, then $\varphi(f(u)) = f(au)$ and $\psi(f(u)) = f(bu)$. Applying both these identities, we get: 
$$
(\psi\circ\varphi)(f(u)) = f(abu) = (\varphi\circ\psi)(f(u)).
$$
If $X$ is a constant, then $M$ and $N$ are also constants, and obviously they commute. If $X$ is not a constant, then $f$ takes at least one value $z$ inside the unit disc $\mathbb D = \{z \in \mathbb C\mid |z| < 1\}$, \cite[Chapter 2]{FTbook}, Theorem 2.1.4, Corollary 2. But $f$ is continuous, therefore this value $z$ is the limit point of the image of $f$. Both $\varphi\circ\psi$ and $\psi\circ\varphi$ are probability generating functions of some distributions. Therefore, they are analytic on $\mathbb D$. Applying the classic result from compelx analysis, we get~\eqref{eq:commuting-equation}. 

{\it Step 4.} If $X$ is $N$-stable and $M$-stable, then we have: $\mathbb E[M] \in (1, \infty)$ and $\mathbb E[M\ln M] < \infty$; also $\mathbb E[N] \in (1, \infty)$ and $\mathbb E[N\ln N] < \infty$.

{\it Step 5.} Finally, the last claim: For distinct $M$ and $N$, the representation $f = L\circ g$ has the same $L$ and the same $g$. This follows from the uniqueness Lemma~\ref{lemma:unique}. In particular, the index $\alpha$ for both $M$ and $N$ is the same. 

\subsection{Proof of Lemma~\ref{lemma:unique}} {\it Step 1.} It is clear from the representation of $g_j \in \mathcal G$ in~\eqref{eq:G}:
\begin{equation}
\label{eq:G-j}
g_j(u) = |u|^{\alpha_j}(\beta_j + \gamma_j\mathrm{i}\,\mathrm{sgn}(u))
\end{equation}
that  $\mathrm{Re}g_j(u) > 0$ for $u \ne 0$; and $g_j(u) \to 0$ as $u \to 0$. And by Lemma~\ref{lemma:tech} 
\begin{equation}
\label{eq:derivative}
\frac{1 - L_j(g_j(u))}{g_j(u)} \to 1,\quad j = 1, 2,\, u \to 0.
\end{equation}
Indeed, $L_j$ is the Laplace transform of a probability measure on $[0, \infty)$ with mean $1$. Next, 
\begin{equation}
\label{eq:equality}
L_1(g_1(u)) = L_2(g_2(u)).
\end{equation}
Comparing~\eqref{eq:derivative} with~\eqref{eq:equality}, we get: $g_1(u)/g_2(u) \to 1$ as $u \to 0$.  

{\it Step 2.} From~\eqref{eq:G-j}, rewrite for $u > 0$:
\begin{equation}
\label{eq:ratio}
\frac{g_1(u)}{g_2(u)} = u^{\alpha_1 - \alpha_2}\cdot\frac{\beta_1 + \gamma_1\mathrm{i}}{\beta_2 + \gamma_2\mathrm{i}}.
\end{equation}
Since $\beta_1, \beta_2 > 0$, then 
$$
D := \left|\frac{\beta_1 + \gamma_1\mathrm{i}}{\beta_2 + \gamma_2\mathrm{i}}\right| > 0.
$$
Taking the limit as $u \to \infty$, yields $\alpha_1 = \alpha_2$, which consequently implies $\beta_1 = \beta_2$ and $\gamma_1 = \gamma_2$. Thus $g_1 \equiv g_2$. 

{\it Step 3.} To complete the final step and show $L_1 \equiv L_2$, recall $L_1(g_1(u)) \equiv L_2(g_2(u))$. The functions $L_1$ and $L_2$ are analytic on the half-space $\mathbb H := \{z \in \mathbb C\mid \mathrm{Re}(z) > 0\}$, by Lemma~\ref{lemma:tech}. And the set $\{g(u)\mid u > 0\}$ has limit points in $\mathbb H$. By the classic uniqueness result from complex analysis, $L_1 \equiv L_2$. This complets the proof of Lemma~\ref{lemma:unique}. 

\subsection{Proof of Lemma~\ref{lemma:symmetry}} This is another immediate consequence of Theorem~\ref{thm:repr}: 
$$
X \stackrel{d}{=} -X\, \Rightarrow\, Z \stackrel{d}{=} -Z.
$$
The function $g$ from~\eqref{eq:G} is even: $g(u) = g(-u)$. This happens if and only if $\gamma = 0$, so $g(u) = \sigma |u|^{\alpha}$. 

\subsection{Proof of Lemma~\ref{lemma:Gaussian}} From Theorem~\ref{thm:repr}, we get: $X = \sqrt{Y}Z$, where $Z$ is strictly $2$-stable and therefore Gaussian. Thus 
$\mathbb E[X^2] = \mathbb E[Y]\cdot\mathbb E[Z^2] < \infty$, since 
$\mathbb E[Y] = 1$, which follows from $\mathbb E[N\ln N] < \infty$. 

\subsection{Proof of Lemma~\ref{lemma:positive}} Since $X \ge 0$, by assumption, then the strictly stable random variable $Z$ with index $\alpha$ must be nonnegative as well. But this could be true only if $\alpha = 1$ and $Z$ is a positive constant; or $\alpha \in (0, 1)$, and $Z$ has Laplace transform~\eqref{eq:Laplace-alpha}. This can be found in \cite[Chapter V]{SteutelBook}, Theorem 3.5, or in the classic reference \cite{STbook}. Then the Laplace transform of $X$ is 
\begin{align*}
\mathbb E[e^{-uX}] &= \mathbb E\bigl[\exp\bigl(-uY^{1/\alpha}Z\bigr)\bigr] = \mathbb E\,\mathbb E\left[\exp\bigl(-uY^{1/\alpha}Z\bigr)\mid Y\right] \\ & = \mathbb E\left[\exp\bigl(-\sigma u^{\alpha}Y\bigr)\right] = L(\sigma u^{\alpha}).
\end{align*}

\subsection{Proof of Lemma~\ref{lemma:inf-div}} Using the notation from the beginning of this article, we see that the PGF of $M$ is $\varphi^{1/k}$. Therefore, the PGF of $kM$ is $\psi(s) = \varphi^{1/k}(s^k)$. The Laplace transform of $Y$ is $L^{1/k}$. It is straightforward to check that~\eqref{eq:main-ChF} holds if and only if $\psi(L^{1/k}(u)) = L^{1/k}(cu)$ for $u \ge 0$. This completes the proof. 

\subsection{Proof of Lemma~\ref{lemma:curious}} This follows from two consecutive applications of Theorem~\ref{thm:commute}. Indeed, if $X$ is $N$-stable and $M$-stable, then $M$ and $N$ commute. But $Y$ is $N$-stable, and therefore $Y$ is $M$-stable. 

\subsection{Proof of Lemma~\ref{lemma:Gauss}} We express $L(u) = (1 + u)^{-\alpha}$ for the Gamma random variable $X$ with shape $\alpha$ and scale $1$. We find the inverse function $L^{\leftarrow}(s) = s^{-1/\alpha} - 1$. For $c > 1$, we have from~\eqref{eq:Bunge}:
$$
\varphi(u) = u\left(u^{1/\alpha}(1-c) + c\right)^{-\alpha}.
$$
It has Taylor decomposition which contains $u^{1 + 1/\alpha}$. This exponent is an integer if and only if $\alpha = 1/n$ for $n = 1, 2, 3, \ldots$ 

\subsection{Proof of Lemma~\ref{lemma:limit-point}} {\it Step 1.} If $1$ is the limit point of $S(Q)$, then for every interval $(a, b) \subseteq (1, \infty)$, however small, we have a $c \in S(P) \cap (1, b/a)$. Then for some $n$, we have $c^n \in (a, b)$. This $c^n \in S(Q)$, which proves $S(Q)$ is dense in $[1, \infty)$: It is intersecting with any interval. Since $S(Q)$ is topologically closed, $S(Q) = [1, \infty)$. 

{\it Step 2.} All other claims about $S(Q)$ and $G(Q)$ directly follow from \cite[Proposition 1]{Bunge}. They use the notation $e^t = c$ where $t \ge 0$ and $c \ge 1$. 

\subsection{Proof of Lemma~\ref{lemma:finite-moments}} This was, in fact, already proven in the literature, since the tails of standard $N$-stable $Y$ are related to tails of $N$: It is shown in \cite[Chapter I]{AthreyaBook}, Section 10, Theorem 2, that for any $a \ge 0$, 
$$
\mathbb E[N\ln^{1 + a}(N)] < \infty \Leftrightarrow \mathbb E[Y\ln^a Y] < \infty.
$$
Also, it is shown in \cite[Theorem 0]{Bingham1974} that for $k = 2, 3, \ldots$, we have: 
$$
\mathbb E[N^k] < \infty \Leftrightarrow \mathbb E[Y^k] < \infty.
$$

\subsection{Proof of Lemma~\ref{lemma:semigroup-scaling}} {\it Step 1.} Assume 
$\mathbb E[Y] = 1$. The Laplace transform of $N_c/c$:  
\begin{equation}
\label{eq:laplace-semigroup}
\mathbb E\left[e^{-N_cu/c}\right] = L\left(cL^{\leftarrow}\left(e^{-u/c}\right)\right).
\end{equation}
If $Y$ has mean $1$, then $L'(0+) = -1$. Derivative of the inverse function: $(L^{\leftarrow})'(1-) = -1$. Also, an elementary calculus result shows:
$$
e^{-u/c} - 1 \sim \frac{-u}{c},\quad c \to \infty.
$$
Combining these asymptotics, we get:
\begin{equation}
\label{eq:laplace-conv}
cL^{\leftarrow}\left(e^{-u/c}\right) \to (L^{\leftarrow})'(1-) \cdot c\cdot\left(-\frac{u}{c}\right) = u. 
\end{equation}
Applying~\eqref{eq:laplace-conv} to~\eqref{eq:laplace-semigroup}, we get that the left-hand side of~\eqref{eq:laplace-semigroup} converges to $L(u)$, which is the Laplace transform of $Y$. This completes the if part. 

{\it Step 2.} Conversely, if indeed there is such convergence, it holds for $c = b^n$ as well, where $b \in S(Q)$. But this corresponds to scaling of a discrete-time branching process. Using the aforementioned discussion about parameterization, we see that such  scaling works only when $\mathbb E[Y] < \infty$, which is equivalent to $\mathbb E[Y] = 1$. 

\subsection{Proof of Lemma~\ref{lemma:equiv}} {\it Step 1.} Let us derive ~\eqref{eq:main-RV} from~\eqref{eq:1}. Direct application of \cite[Chapter 4]{GnedenkoBook}, Theorem 4.6.5, with the following notation mathc: $\theta = 1/c$ and $m(\theta) = [c]$, $X_{\theta, j} = U_j/a(c)$, $F$ is the CDF of $X$, and $G$ is the CDF of $Z$, and $N = Y$ in (6.32), and (6.30) from \cite[Theorem 4.6.3]{GnedenkoBook} can be rewritten as~\eqref{eq:repr}. See our notation discussion above. 

{\it Step 2.} Conversely, if $X$ is $N$-stable, then we can simply take $U = X$. Then we have equality in law in~\eqref{eq:1} instead of weak convergence. 

{\it Step 3.} Equivalence of~\eqref{eq:1} and~\eqref{eq:2} also follows from \cite[Chapter 4]{GnedenkoBook}, Theorem 4.6.5.

\subsection{Analytical Laplace transform} Take a random variable $X \ge 0$. Consider its Laplace transform $L(u) = \mathbb E[e^{-uX}]$.

\begin{lemma} Assume $\mathbb E[X] < \infty$. We can define $L$ as a complex-valued function on the half-plane $\mathbb H = \{z \in \mathbb C\mid \mathrm{Re}(z) > 0\}$. This function is analytic on $\mathbb H$, and we have: 
\begin{equation}
\label{eq:deriv}
\frac{1 - L(z)}{z} \to \mathbb E[X],\quad z \to 0,\quad \mathrm{Re}(z) > 0.
\end{equation}
\label{lemma:tech}
\end{lemma}

\begin{proof} For $z = u + \mathrm{i}v \in \mathbb H$, we write $|e^{-zX}| = e^{-uX}$. Therefore, $\mathbb E|e^{-zX}| = \mathbb E[e^{-uX}] < \infty$, and $\mathbb E[e^{-zX}]$ is well defined. It is complex analytic on $\mathbb H$. To this end, we need to prove in the complex analytic sense: $L'(z) = \mathbb E[-Xe^{-zX}]$. But this, in turn, can be proved as follows. Take two points $z, z' \in \mathbb H$ and draw a segment between them. It is parametrized as $w_t = zt + z'(1-t)$. We need to show 
\begin{equation}
\label{eq:FTC}
L(w_t) - L(z) = \int_{[z, w_t]}\mathbb E\left[-Xe^{-uX}\right]\,\mathrm{d}u.
\end{equation}
This integral over the segment is understood in the complex analytic sense. This can be written using real-valued integration as
$$
\int_0^t\mathbb E\left[-Xe^{-(zs + (1-s)z')X}\right](z' - z)\,\mathrm{d}s.
$$
Remove expectations for a moment. From complex analysis, we get:
$$
e^{-(zt + z'(1-t))X} - e^{-ztX} = \int_0^t\left(-Xe^{-w_sX}\right)(z' - z)\,\mathrm{d}s.
$$
We need only to show that we need an interchange of integration and expectation. This requires us to use the Fubini theorem. Usually, this theorem is stated for real-vaued functions, but it works equally well for complex-valued functions. We need:
\begin{equation}
\label{eq:Fubini}
\mathbb E\int_0^t\left|-Xe^{-w_sX}\right|(z'-z)\,\mathrm{d}s < \infty.
\end{equation}
Assuming $\mathrm{Re}(z) \le \mathrm{Re}(z')$ without loss of generality, 
$\mathrm{Re}(z) \le \mathrm{Re}(w_s)$ for all $s \in [0, 1]$. Hence
$$
\mathbb E|-Xe^{-w_sX}| = \mathbb E|Xe^{-\mathrm{Re}(z)X}| < \infty.
$$
This proves~\eqref{eq:Fubini}, and with it proves~\eqref{eq:FTC}. Thus $L$ is analytic on $\mathbb H$. The property~\eqref{eq:deriv} can be shown similarly, with $z = 0$, since $\mathbb E[Xe^{-zX}] = \mathbb E[X]$ for $z = 0$. 
\end{proof}

\end{document}